\documentclass[12 pt,twoside]{amsart}
\usepackage{amsmath}
\usepackage{amssymb}
\usepackage{hyperref}
 \usepackage{amsthm}

\begin{document}

\newtheorem{theorem}{Theorem}[section]
\newtheorem{corollary}[theorem]{Corollary}
\newtheorem{proposition}[theorem]{Proposition}
\newtheorem{lemma}[theorem]{Lemma}
\theoremstyle{definition}
\newtheorem{definition}[theorem]{Definition}
\theoremstyle{remark}
\newtheorem{remark}[theorem]{\bf Remark}
\newcommand{\ds}{\displaystyle}
\newcommand{\R}{\mathbb{R}}
\newcommand{\noi}{\noindent}
\newcommand{\ts}{\thinspace}
\newcommand{\leqn}[1]{\\*\\* \indent #1 \\*\\*}
\newcommand{\leqnn}[1]{\\*\\* \indent #1 \\*}
\newcommand{\bleqn}[1]{\\*\\*\\* \indent #1 \\*}
\renewcommand{\thesubsection}{}
\newcommand{\norm}[1]{\left \lVert #1 \right \rVert}
\newcommand{\biarraysmall}[2]{\begin{array}{c} #1 \vspace{-.20cm}\\ #2  \end{array}}
\newcommand{\biarray}[2]{\begin{array}{c} #1 \vspace{.3cm}\\ #2  \end{array}}
\newcommand{\triarray}[3]{\begin{array}{c} #1 \vspace{.15cm}\\ #2 \vspace{.15cm}\\ #3 \end{array}}
\newcommand{\quadarray}[4]{\begin{array}{c} #1 \vspace{.15cm}\\ #2 \vspace{.15cm}\\ #3 \vspace{.15cm}\\ #4 \end{array}}
\newcommand{\quintarray}[5]{\begin{array}{c} #1 \vspace{.15cm}\\ #2 \vspace{.15cm}\\ #3 \vspace{.15cm}\\ #4 \vspace{.15cm}\\ #5 \end{array}}
\newcommand{\barr}{\begin{eqnarray}}
\newcommand{\beq}{\begin{equation}}
\newcommand{\bpf}{\begin{proof} \quad}
\newcommand{\btm}{\begin{thm}}
\newcommand{\blem}{\begin{lem}}
\newcommand{\elem}{\end{lem}}
\newcommand{\earr}{\end{eqnarray}}
\newcommand{\eeq}{\end{equation}}
\newcommand{\epf}{\end{proof}}
\newcommand{\etm}{\end{thm}}
\newcommand{\beqq}{\begin{equation*}}
\newcommand{\eeqq}{\end{equation*}}

%

\title[2nd order periodic equations]{Periodic behaviour of nonlinear second order discrete dynamical systems}
\author[Daniel Maroncelli and J. Rodr\'{\i}guez ]{  Daniel Maroncelli$^*$ and Jes\'{u}s Rodr\'{\i}guez\\ }
\address[D. Maroncelli]{Department of Mathematics, Box 7388, Wake Forest University, Winston-Salem, NC 27109-6233, U.S.A}
\address[J. Rodr\'{\i}guez]{Department of Mathematics, Box 8205, North Carolina State University, Raleigh, NC 27695-8205, U.S.A}
\email[D. Maroncelli]{maroncdm@wfu.edu}
\email[J.  Rodr\'{\i}guez]{rodrigu@ncsu.edu}
\thanks{$^*$Corresponding Author}

{\footnotesize
\noindent
 \begin{abstract}   In this work we provide conditions for the existence of periodic solutions to nonlinear, second-order difference equations of the form
\begin{equation*}
y(t+2)+by(t+1)+cy(t)=g(t,y(t))
\end{equation*}
where $c\neq 0$, and $g:\mathbb{Z}^+\times\R\to \R$ is continuous and periodic in $t$. Our analysis uses the Lyapunov-Schmidt reduction in combination with fixed point methods and topological degree theory.
\end{abstract}

\maketitle
\begin{center}\textbf{{\tiny($^{**}$The final version of this manuscript has been accepted for publication in Journal of Difference Equations and Applications)}}\end{center}



\section{Introduction}
\setlength{\parskip}{.25cm plus 0.05cm minus 0.02cm}

In this work we provide conditions for the existence of periodic solutions to nonlinear, second-order difference equations of the form
\begin{equation}\label{periodicdiff}
y(t+2)+by(t+1)+cy(t)=g(t,y(t)).
\end{equation}
Throughout our discussion we will assume that $b$ and $c$ are real constants, $c\neq 0$, and  $g:\mathbb{Z}^+\times\R\to\R$ is continuous and periodic in its first component.

In \cite{Etheridgesecondorder}, Rodr\'{\i}guez and Etheridge prove the existence of $N$-periodic solutions to (\ref{periodicdiff}) under the following conditions:
\begin{enumerate}
\item[H1)] \label{rodhyp1}The nonlinearity $g$ is independent of $t$ and sublinear; that is, there exists $M_1$, $M_2$ and $\beta$, with $0\leq \beta<1$,  such that for all $x\in \R$, $|g(x)|\leq M_1|x|^\beta+M_2$.
\item[H2)\label{rod}] \label{rodhyp2}There exists a constant $\hat{z}$ such that for all $x$ with $|x|\geq \hat{z}$, $xg(x)>0$.
\item[H3)] $N>1$ is odd, with $N\arccos\left(-{b\over 2}\right)$ not a multiple of $2\pi$ whenever $c=1$ and $|b|<2$.
\end{enumerate}
Their analysis was highly dependent on the structure of the solution space to the periodic linear homogeneous problem
\begin{equation}\label{periodicdiffhom}
y(t+2)+by(t+1)+cy(t)=0
\end{equation}
and its interaction with the nonlinearity $g$.

In this paper we extend the ideas of \cite{Etheridgesecondorder} to a more general class of nonlinearities.  We also obtain results in the case where $N\arccos\left(-{b\over 2}\right)$ is a multiple of $2\pi$.
In this case, the solution space of the periodic problem, (\ref{periodicdiffhom}), is two-dimensional causing the analysis of the nonlinear problem to be much more delicate than the cases observed in  \cite{Etheridgesecondorder}.
Under the conditions given in  \cite{Etheridgesecondorder}, the solution space of (\ref{periodicdiffhom}) is either trivial or one-dimensional.  When the solution space of (\ref{periodicdiffhom}) is two-dimensional, the analysis of (\ref{periodicdiff})  is more complex
due to the fact that the interaction of the solution space to this associated homogeneous problem and the nonlinearity, $g$,  is much more complicated.

The study of nonlinear boundary value problems for finite difference equations  is extensive. For those interested, we mention just a few. In \cite{EthridgePeriodic,Etheridgesecondorder,Zhou},  periodic solutions are analyzed. In \cite{MaSturm,RodSturm,RodAbSturmGlobal} the authors study existence of solutions to nonlinear discrete Sturm-Liouville problems.  \cite{MaroncelliDiscreteBounded, RodTaylorscalardiscrete, ZhangMultipoint} discuss the existence of solutions to multi-point problems. Positive solutions are discussed in \cite{HendersonThreepoint,HendersonPost, HendersonPostMult}. Results regarding the existence of multiple solutions may be found in \cite{HendersonThompson,LiuMult}


 \section{Preliminaries}

 The nonlinear boundary value problem, (\ref{periodicdiff}), will be viewed as an operator problem. We start by defining
\begin{equation*}
A=
\begin{pmatrix}
0&1\\
-c&-b
\end{pmatrix},
\end{equation*}
 and 
 $f:\mathbb{Z}^+\times \R^2\to{\R}^2 $ by 
\begin{equation*}
f(t,(u,v))=\begin{pmatrix}0\\g(t,u)\end{pmatrix}.
\end{equation*}

Using the standard procedure of writing a scalar difference equation as system, then provided $g$ is $N$-periodic in $t$, the nonlinear boundary value problem, (\ref{periodicdiff}), with periodicity of period $N$, is equivalent to the nonlinear system
\begin{equation}
 x(t+1)=Ax(t)+f(t,x(t)) \hspace{.25cm}  \label{E:System} 
 \end{equation}                          
 subject to 
 \begin{equation}
 x(0)-x(N)=0,
 \end{equation}
where $x(t)$ denotes $\ds{\begin{pmatrix}y(t)\\y(t+1)\end{pmatrix}}$.

The underlying function space for our operator problem is
\begin{equation*}
 \ds{X_N=\left\{\phi:\mathbb{Z}^+\to {\R}^2 \left| \phi \text{ is } N\text{-periodic}\right.\right\},}
\end{equation*}
with the topology of $X_N$ being that generated by the supremum norm. We use $\norm{\cdot}$ to denote this norm, and we will use $|\cdot|$ to denote the standard Euclidean norm.
 
We define operators as follows:

 $\mathcal{L}\colon X_N\to X_N$  by
 \begin{equation*}
(\mathcal{L}x)(t)=x(t+1)- Ax(t)\label{O:L},
\end{equation*}
and

$\mathcal{F}\colon X_N \to X_N$ by
\begin{equation*}\mathcal{F}(x)(t)=f(t,x(t)).
\end{equation*}

\noindent Solving the nonlinear boundary value problem (\ref{periodicdiff}) with periodicity of period $N$ is now equivalent to solving 
\begin{equation}
\mathcal{L}x=\mathcal{F}(x).\label{E:SOper}
\end{equation}


We begin our study of the nonlinear boundary value problem, (\ref{periodicdiff}), with an analysis of the linear nonhomogeneous problem $\mathcal{L}x=h$, where $h$ is a $N$-periodic function.
 
It is clear that the principal fundamental matrix solution to 
 \begin{equation}
  x(t+1)=Ax(t)\label{A:Homog}.
  \end{equation} 
  is given by $A^t$.  We will use this fact to completely characterize the solutions to the linear nonhomogeneous problem $\mathcal{L}x=h$.
For those readers interested in the general theory of difference equations, we suggest \cite{Elaydi, Kelley}. \cite{Kelley} has a nice introduction to second-order self-adjoint difference equations, regular Sturm-Liouville problems being a special case.  
\cite{Elaydi} may be useful to those interested in periodic linear systems.
\begin{proposition}
\label{P:Perp}An element $h\in X_N$  is contained in the $Im(\mathcal{L})$ if and only if 
\begin{equation*}A^N\sum_{i=0}^{N-1}A^{-{(i+1)}}h(i)\in Ker\left(\left(I-A^N\right)^T\right)^\perp.\end{equation*} 
 \end{proposition}

\begin{proof}
If $\mathcal{L}x=h$, for some $x\in X_N$, then using the variation of parameters formula, we have
\begin{equation*}
 x(t)=A^tx(0)+A^t\sum_{i=0}^{t-1}A^{-{(i+1)}}h(i)\label{E:Var}.
\end{equation*}
 Now, using the fact that $x(0)=x(N)$, we have that 
 
 \begin{equation*}
 x(0)=A^Nx(0)+A^N\sum_{i=0}^{N-1}A^{-{(i+1)}}h(i)\label{E:Var}.
\end{equation*}
Thus, $\mathcal{L}x=h$ if and only if $\ds{A^N\sum_{i=0}^{N-1}A^{-{(i+1)}}h(i)\in Im(I-A^N)}$.
The result now follows from the fact that  $Im(I-A^N)=Ker\left(\left(I-A^N\right)^T\right)^\perp.$
\end{proof}

It follows from Proposition $\ref{P:Perp}$ that we may choose a matrix, $W$, whose columns form a basis for $Ker\left((I-A^N)^T\right)$, such that $h\in Im(\mathcal{L})$ if and only 
\begin{equation*}
W^TA^N\sum_{i=0}^{N-1}A^{-{(i+1)}}h(i)=0.
\end{equation*}
If for each $t\in Z^+$ we define $\ds{\Psi(t)=\begin{cases} {(A^N)}^TW& t=0\\ {(A^{-{(i+1)}})}^T{{(A^N)}^T} W& t>0\end{cases}}$, then it is clear that $\mathcal{L}x=h$ if and only if $\ds{\sum_{i=0}^{N-1}\Psi^T(i+1)h(i)}=0$. 

It will be useful to know, see \cite{Etheridgesecondorder} for the proof, that the columns of $\Psi$ span the solution space of the $N$-periodic  linear homogeneous ``adjoint'' problem
\begin{equation}
 \mathcal{L}^*x=0, \label{E:SystemAd} 
 \end{equation} 
 where $\mathcal{L}^*:X_N\to X_N$ is defined by
 \begin{equation*}
 (\mathcal{L}^*x)(t)=x(t+1)-A^{-T}x(t).
 \end{equation*} 
 Since any fundamental matrix solution to the ``adjoint'' problem, (\ref{E:SystemAd}), is of the form $\Psi(t)C$, for some invertible matrix $C$, we have that  $\mathcal{L}x=h$ if and only if 
 \begin{equation*}
 \ds{\sum_{i=0}^{N-1}\Gamma^T(i+1)h(i)}=0
 \end{equation*}
 for any fundamental matrix solution to (\ref{E:SystemAd}), $\Gamma$.

It will also be useful to know that the ``adjoint'' system produces periodic solutions to a scalar difference equation which is very closely related to (\ref{periodicdiff}).   The details are as follows.  Calculating $A^{-T}$, we get
 \begin{equation*}
 \ds{A^{-T}={1\over c}\begin{pmatrix}-b&c\\-1& 0\end{pmatrix}}.
 \end{equation*}
 The ``adjoint'' system, \ref{E:SystemAd}, is now equivalent to
 \begin{equation*}
 \begin{split}
 cx_1(t+1)&=-bx_1(t)+cx_2(t)\\
 cx_2(t+1)&=-x_1(t)
 \end{split}.
 \end{equation*} 
This gives 
\begin{equation*}
cx_2(t+2)+bx_2(t+1)+x_2(t)=0,
\end{equation*}
so that the second component of  a solution to the  ``adjoint'' system is a $N$-periodic solution to  
\begin{equation}
cy(t+2)+by(t+1)+y(t)=0.\label{periodicdiffad}
\end{equation}
We also have the following result regarding the linear homogeneous system, (\ref{A:Homog}).
\begin{proposition}
\label{P:KerL}The solution space of the linear homogeneous problem (\ref{A:Homog}) and the $\ds{ Ker\left(I-A^N\right)}$ have the same dimension.
\end{proposition}

\begin{proof}
 It is clear that $\mathcal{L}x=0$ if and only if  $\exists v\in {\R}^2$ satisfying $x(t)=A^tv$ and $ (I-A^N)v=0$. \end{proof}

We will analyze the nonlinear periodic problem, (\ref{periodicdiff}), using an alternative method  in conjunction with fixed points methods and topological degree theory. Crucial to the use of this alternative method is the construction of projections onto the kernel and image of $\mathcal{L}$.  We choose to follow \cite{Etheridgesecondorder}.
\begin{proposition}\label{P:ProjectionKer}
Let $V$ be the orthogonal projection onto $ker(I-A^N)$. If we define $P:X_N\to X_N$ by $(Px)(t)=A^tVx(0)$, then $P$ is a projection onto the $Ker(\mathcal{L})$.
\end{proposition}
\begin{proposition}\label{P:ProjectionIm}
If we define $Q:Z\to Z$ by 

\begin{equation*}
\begin{split}
(Qh)(t)&=\Psi(t)\left(\sum_{j=0}^{N-1}|\Psi(j)|^2\right)^{-1}\sum_{i=0}^{N-1}\Psi^T(i)h(i)\\
\end{split}
\end{equation*},

  then $Q$ is a projection with $Ker(Q)=Im(\mathcal{L})$.
\end{proposition}


The following is a formulation of the alternative method, often referred to as the  Lyapunov-Schmidt procedure, which we will use to analyze the nonlinear problem, (\ref{periodicdiff}).  Interested readers may consult \cite{Chow,MaroncelliNonlinearImp}.  We include the proof for the benefit of the reader.
 
 \begin{proposition}\label{P:Lyap} Solving $\mathcal{L} x=\mathcal{F}(x)$ is equivalent to solving the system
 \begin{equation*}\left\{\begin{array}{c}
 x=Px+M_p(I-Q)\mathcal{F}(x) \\
\text{and}\\
 Q\mathcal{F}(x)=0
 \end{array}
 \right.
 \end{equation*}
 where $M_p$ is $\ds{(\mathcal{L}{_{|Ker(P)\cap dom(\mathcal{L})}})^{-1}}$.
 \end{proposition}
 
 \begin{proof}
\begin{equation*}
\begin{split}
 \mathcal{L} x=\mathcal{F}(x)&\Longleftrightarrow \left\{\begin{array}{c} (I-Q)(\mathcal{L} x-\mathcal{F}(x))=0 \\ \text{ and }\\Q(\mathcal{L} x-\mathcal{F}(x))=0 \end{array} \right. \\
 \\
& \Longleftrightarrow \left\{ \begin{array}{c} \mathcal{L} x-(I-Q)\mathcal{F}(x)=0 \\ \text{ and }\\Q\mathcal{F}(x)=0 \end{array} \right.\\
\\
&\Longleftrightarrow \left\{\begin{array}{c} M_p\mathcal{L} x-M_p(I-Q)\mathcal{F}(x)=0  \\ \text{ and }\\ Q\mathcal{F}(x)=0\end{array} \right.\\ 
\\
&\Longleftrightarrow\left\{\begin{array}{c}  {(I-P)x-M_p(I-Q)\mathcal{F}(x)=0}\\ \text{ and }\\Q\mathcal{F}(x)=0 \end{array} \right.\\
\end{split}
\end{equation*}
\end{proof}


\section{$dim(Ker(\mathcal{L}))<2$}
We now come to our first existence theorem. Throughout this section we will assume that $N>1$ is odd, and  that $g$ is $N$-periodic in $t$.

 \begin{theorem}\label{T:Main1} Let $r, \hat{z}$, and $\delta$ be positive constants and suppose the following conditions hold.

 \begin{enumerate}
 
  \item[C1.] $|g(t,x)|\leq \delta$  for all $t\in\mathbb{Z}^+$ and all $x\in[-2r,2r]$.
     \vspace{.25cm}
  
  \item[C2.] $xg(t,x)> 0$ for all $t\in\mathbb{Z}^+$ and $|x|>\hat{z}$. 
  \vspace{.25cm}
  
   \item[C3.] $\hat{z}+\norm{M_p(I-Q)}\delta<r$.
   \vspace{.25cm}
   
   \item[C4.] If $N\arccos(-{b\over2})$ is a multiple of $2\pi$, then $c\neq 1$ or $2\leq|b|$.
 \end{enumerate}

\noindent Then  (\ref{periodicdiff}) has a $N$-periodic solution. 
 \end{theorem}

\begin{proof}

If C4. holds, see \cite{Etheridgesecondorder} for the details, it is easy to show that one of the following occurs

 \begin{enumerate}
 \item[I.]  The $Ker(\mathcal{L})$ is trivial.
  \item[II.] $\mathcal{L}$ and $\mathcal{L}^*$ are singular with $Ker(\mathcal{L})$ being spanned by $\ds{\begin{pmatrix}1\\1\end{pmatrix}}$ and the $Ker(\mathcal{L}^*)$  being spanned by $\ds{\begin{pmatrix}-c\\1\end{pmatrix}}$.
  \end{enumerate}

We start by assuming $\mathcal{L}$ is invertible.  Since $\mathcal{L}$ is invertible, it is clear   that  the projections $P$ and $Q$ are the zero projections on $X_N$ and that $M_p$ is $\mathcal{L}^{-1}$. In this case, solving (\ref{E:SOper}) reduces to finding the fixed points of $H_1:=\mathcal{L}^{-1}\mathcal{F}$.  We will show $H_1$ has a fixed point using the Brouwer fixed point theorem.

Define $\Omega_1$ to be the closed ball of radius $r$ in $X_N$.  For $x\in\Omega_1$, we have
\begin{equation*}
\begin{split}
\norm{H_1(x)}&=\norm{\mathcal{L}^{-1}\mathcal{F}(x)}\\
	         &\leq \norm{\mathcal{L}^{-1}}\norm{\mathcal{F}(x)}\\
	         &\leq \norm{\mathcal{L}^{-1}}\delta\hspace{1.5cm}\text {(Using C1.)}\\
	         &= \norm{M_p(I-Q)}\delta\\
	         &< r. \hspace{2.5cm}\text {(Using C3.)}
\end{split}
\end{equation*}
Thus, $H_1$ maps $\Omega_1$ into itself and therefore has a fixed point.

We now turn our attention to the case when  $Ker(\mathcal{L})$ is spanned by $\ds{\begin{pmatrix}1\\1\end{pmatrix}}$.
As a matter of notation, we introduce, for $\alpha\in\R$ and $x\in  Im(I-P)$,
\begin{equation*}
p(\alpha,x)=M_p(I-Q)\mathcal{F}\left(\alpha\begin{pmatrix}1\\1\end{pmatrix}+x\right).
\end{equation*}
We then define $H_2:\R\times Im(I-P)\to \R\times Im(I-P)$ as follows
\begin{equation}\label{F:Fixed}
H_2(\alpha,x)=\begin{pmatrix}
\ds{\sum_{i=0}^{N-1}g(i,\alpha +h(\alpha,x)(i))}
\\
x-p(\alpha,x)
\end{pmatrix},
\end{equation}
where $h(\alpha, x)$ denotes the first component of $p(\alpha,x)$.
It is clear from Proposition $\ref{P:Lyap}$ and the descriptions of the $Ker(\mathcal{L})$ and $Ker(\mathcal{L}^*)$ given above, that the solutions to 
(\ref{E:SOper}) are the zeros of $H_2$.

We make $\R\times Im(I-P)$ a Banach space using the max norm; that is,  
\begin{equation*}
\norm{(\alpha,x)}=\max\{|\alpha|,\norm{x}\}
\end{equation*}
 and let $\Omega_2$ denote the closed ball of radius $r$ in $\R\times Im(I-P)$ under this norm.
We will show $H_2$ has a 0 by showing the Brouwer degree, $deg(H_2,\Omega_2,0)$, is nonzero.

Define $E:[0,1]\times \overline{\Omega_2}\to  \R\times Im(I-P)$ by 
\begin{equation*}
E(\lambda,(\alpha,x))=\begin{pmatrix}
\ds{(1-\lambda)\alpha+\lambda\sum_{i=0}^{N-1}g(i,\alpha +h(\alpha,x)(i))}
\\
x-\lambda p(\alpha,x)
\end{pmatrix},
\end{equation*}
It is evident that $E$ is a homotopy between the identity and $H_2$. In what follows, we will show that  $E(\lambda, (\alpha,x))$ is nonzero for each $\lambda\in(0,1)$ and every $(\alpha,x)\in \partial(\Omega_2)$, so that $deg(H_2,\Omega_2,0)\neq 0$ by the invariance of the Brouwer degree under homotopy.

We first observe that for every  $(\alpha,x)\in \Omega_2$, $\norm{p(\alpha, x)}<r$.
To see this note that $|\alpha+x_1(i)|\leq |\alpha|+|x_1(i)|\leq |\alpha|+\norm{x}\leq 2r$, where $x_1$ is the first component of $x$. Therefore, 
\begin{equation*}
\begin{split}\label{L:1Res}
\norm{p(\alpha,x)}&=\norm{M_p(I-Q)\mathcal{F}\left(\alpha\begin{pmatrix}1\\1\end{pmatrix}+x\right)}\\
		       &\leq\norm{M_p(I-Q)}\norm{\mathcal{F}\left(\alpha\begin{pmatrix}1\\1\end{pmatrix}+x\right)}\\
		       &=\norm{M_p(I-Q)}\sup_i|g(i, \alpha+x_1(i))|\\
		       &\leq\norm{M_p(I-Q)}\delta<r.		       
\end{split}
\end{equation*}
Thus, if $(\alpha,x)\in \partial(\Omega_2)$ with $\norm{x}=r$, then $x-\lambda p(\alpha,x)\neq0$ for each $\lambda\in(0,1)$.

Further, for each $(\alpha,x)\in \Omega_2$, if $\alpha\geq \hat{z}+\norm{h(\alpha,x)}$, then for $i=0, 1, \cdots N-1$, we have $g(i,\alpha+h(\alpha, x)(i))>0$.  However, since
$\norm{h(\alpha,x)}\leq\norm{p(\alpha,x)}\leq \norm{M_p(I-Q)}\delta$, then provided $\alpha=r$, we have, using C3., that this is the case.  Thus, when $\alpha=r$, $\alpha$ and $\ds{\sum_{i=0}^{N-1}g(i,\alpha +h(\alpha,x)(i))}$ are both postive. A similar argument shows that when $\alpha=-r$,
$\alpha$ and $\ds{\sum_{i=0}^{N-1}g(i,\alpha +h(\alpha,x)(i))}$ are both negative. Using this observation and the fact that 
$\ds{(1-\lambda)\alpha+\lambda\sum_{i=0}^{N-1}g(i,\alpha +h(\alpha,x)(i))}$
 would be zero for some $\lambda\in(0,1)$ if and only if $\alpha$ and $\ds{\sum_{i=0}^{N-1}g(i,\alpha +h(\alpha,x)(i))}$ have opposite sign, we conclude that $\ds{(1-\lambda)\alpha+\lambda\sum_{i=0}^{N-1}g(i,\alpha +h(\alpha,x)(i))}\neq 0$ whenever $(\alpha,x)\in\Omega_2$ and $|\alpha|=r$. 

Since for each $(\alpha,x)\in\partial(\Omega_2)$ either $|\alpha|=r$ or $\norm{x}=r$, we have, from our discussion above, that for each $\lambda\in(0,1)$, $E(\lambda, (\alpha,x))\neq 0$. Thus, $deg(H_2,\Omega_2,0)=deg(I,\Omega_2,0)=1$ and the result follows.
\end{proof}

\begin{remark} We would like to point out that if the nonlinearity $g$ is $N$-periodic in $t$ and sublinear in $x$, as is assumed in \cite{Etheridgesecondorder}, then provided C2. holds, the conditions of Theorem \ref{T:Main1} are satisfied.  To see this note that if $g$ is sublinear in $x$, then there exists $M_1$, $M_2$ and $\beta$, with $0\leq \beta<1$,  such that for all  $t\in\mathbb{ Z}^+$ and $x\in \R$, $|g(t,x)|\leq M_1|x|^\beta+M_2$.  It follows that  $|g(t,x)|<M_1(2r)^\beta+M_2$ for every $t\in \mathbb{Z}^+$ and each $x\in[-2r,2r]$, so that  by choosing $r$ large enough
\begin{equation*}
\hat{z}+\norm{M_p(I-Q)}\left(M_1(2r)^\beta+M_2\right) <r.
\end{equation*}
Thus, Theorem \ref{T:Main1} contains the results found in Theorem 3.4 of \cite{Etheridgesecondorder}.

As the following example demonstrates,  Theorem \ref{T:Main1} is much stronger than the corresponding result found in \cite{Etheridgesecondorder}, Theorem 3.4,  by allowing for more general nonlinearities. If $g$ satisfies condition $C2.$ of  Theorem \ref{T:Main1} and  is $N$-periodic in $t$ with $|g(t,x)|\leq M_1|x|^s+M_2$, for each $t\in\mathbb{ Z}^+$ and every $x\in \R$, where $s\geq 1$, then 
$|g(t,x)|<M_1(2r)^s+M_2$ for every $t\in \mathbb{Z}^+$ and each $x\in[-2r,2r]$.  If we take $r=2\hat{z}$, then 
$|g(t,x)|\leq M_14^s\hat{z}^s+M_2.$ By choosing $M_1$ and $M_2$ small enough, we clearly will have
\begin{equation*}
\hat{z}+\norm{M_p(I-Q)}( M_14^s\hat{z}^s+M_2)<2\hat{z}=r.
\end{equation*}
This example shows that since the growth restrictions of Theorem \ref{T:Main1}  are on bounded intervals of $x$, the nonlinearity, $g$, is allowed to have much more rapid growth in its end behavior than is allowed by Theorem 3.4 of \cite{Etheridgesecondorder}.
\end{remark}

We now state a corollary which gives conditions that ensure the existence of $N$- periodic solutions for all odd integers $N>1$. The result depends on the following: since $\arccos(w)\in [0,\pi]$ for each $w\in[-1,1]$, we have that provided 
\begin{equation*}
\arccos\left(-{b\over 2}\right)\notin Z:=\left\{{2\pi k\over j}\mid k, j\in \mathbb{N} \text{ and } 0\leq 2k<j\right\},
\end{equation*} 
$N\arccos\left(-{b\over 2}\right)$ is not a multiple of $2\pi$ for every odd integer $N>1$. If we define 
\begin{equation*}
U=\left\{b\in(-2,2)\mid \arccos\left(-{b\over 2}\right)\in Z\right\},
\end{equation*}
 we get the following as a consequence of Theorem \ref{T:Main1}.
 
\begin{corollary}\label{C:Main1} Suppose $h:\R\to \R$ is continuous and that the following conditions hold. 
 
 \begin{enumerate}
  \item[C$1^*$.] $\ds{\lim _{r\to \infty}{\norm{h}_{2r}\over r}=0}$, where, for $s>0$, $\norm{h}_s=\sup_{x\in[-s,s]}|h(s)|$.
 \vspace{.25cm}
  
  \item[C$2^*$.]There exists a positive number $R$ such that $xh(x)> 0$ whenever $|x|>R$. 
  \vspace{.25cm}
   
   \item[C$3^*$.] $b\in \R\setminus U$ or $c\neq1$.
 \end{enumerate}

\noindent Then 
\begin{equation}\label{E:Aut}
y(t+2)+by(t+1)+cy(t)=h(y(t))
\end{equation}
 has a $N$-periodic solution for every odd integer $N>1$. 

\end{corollary}
\begin{proof}
Let $N$ denote an odd integer greater than 1.  Define $g(t,x)=h(x)$ for all $t\in\mathbb{Z}^+$ and all $x\in\R$. We will show that the conditions C1.-C4. of Theorem \ref{T:Main1} hold for this choice of $g$. Clearly $g$ is continuous and $N$-periodic in $t$.  Further, C$3^*$. implies $C4.$, so that it suffices to show C1.-C3. Now, using C$1^*.$, we may choose $r^*$ such that if $r\geq r^*$, then 
\begin{equation*}
{R+\norm{M_p(I-Q)}\norm{h}_{2r}\over r}<1. 
\end{equation*}Here $M_p$ and $Q$ are as in Theorem \ref{T:Main1}. If we take $r=r^*$, $\hat{z}=R$ and $\delta=\norm{h}_{2r*}$, then clearly we 
have the following:
\begin{enumerate}
\item[C1.] $|g(t,x)|=|h(x)|\leq\norm{h}_{2r}=\delta$ for all $t\in\mathbb{Z}^+$ and each $x\in[-2r,2r]$.
\item[C2.] $xg(t,x)=xh(x)>0$ for all $t\in\mathbb{Z}^+$ and $|x|>\hat{z}$=R. 
\item[C3.] $\hat{z}+\norm{M_p(I-Q)}\delta=R+\norm{M_p(I-Q)}\norm{h}_{2r}<r$.
\end{enumerate}
The result now follows from Theorem \ref{T:Main1}.
\end{proof}
\begin{remark}
As an example, we have the following: if for all $x\in \R$, $h$ satisfies $\ds{|h(x)|\leq \eta(x)|x|+M_1|x|^\beta+M_2}$, where $0\leq\beta<1$, $\eta:\R\to\R^+$ is continuous, and  $\lim_{x\to \pm\infty}\eta(x)=0$, then  for any $\varepsilon>0$ we eventually have 
\begin{equation*}
\norm{h}_{2r}\leq \varepsilon2r+M_1|2r|^\beta+M_2,
\end{equation*}
so that  $\ds{\lim_{r\to \infty}{\norm{h}_{2r}\over r}=0}$. Thus, provided C$2^*.$ and C$3^*.$ hold, we have $N$-periodic solutions to (\ref{E:Aut}) for all odd integers greater than 1. 

We would like to remark once again that C$1^*.$ is much less restrictive then requiring $h$ to satisfy a sublinear growth condition as is required in Corollary 3.11 of \cite{Etheridgesecondorder}.  This shows, once again, that the results given in this paper generalize those found in \cite{Etheridgesecondorder}.

  As a concrete example, take 
\begin{equation*}\ds{k(x)=\begin{cases}{\ds{-1\over \ln(-x)}}& x\leq-e\\ \ds{{1\over e}}x&-e<x<e\\\ds{{1\over \ln(x)}}&e\leq x\end{cases}},\end{equation*} and define $h(x)=k(x)x+M_1x^\beta+M_2$. From the explanation above, we have that C$1^*.$ holds for this $h$. Further, it is clear that C$2^*.$ holds, so that provided C$3^*$. holds, (\ref{E:Aut}) has a solution. However, $h$  is not sublinear, for if it were, then $|h(x)| <D_1x^\gamma+D_2$, for some real positive real numbers $D_1$,$D_2$, and $\gamma$, with $0<\beta<\gamma<1$.  We would then have that for large $x$,  $\ds{{|x|^{1-\gamma}\over\ln|x|}}$ was  less than 1, which is not the case.
\end{remark}


\section{$dim(Ker(\mathcal{L}))=2$}
We now come to our second existence result. Before giving the formal statement of the theorem, we make some introductory assumptions and then list several useful observations which follow directly from our previous analysis.  

 In what follows we will be assuming that $c=1$, $|b|<2$, $N$ is odd with $N\arccos\left(-{b\over 2}\right)$ a multiple of $2\pi$, and, again, $g$ is $N$-periodic in $t$. In this case, it follows easily that 
\begin{equation*}
\Phi(t)=\begin{pmatrix}
\cos(\theta t) & \sin(\theta t)\\
\cos(\theta (t+1)) & \sin(\theta(t+1))
\end{pmatrix}
\end{equation*}  
is a fundamental matrix solution to (\ref{A:Homog}), where $\theta=\arccos\left(-{b\over 2}\right)$. Since $N\arccos(-{b\over 2})$ is a multiple of $2\pi$, we have that $\Phi$ is $N$-periodic.
Further, since $c=1$, we have that the periodic scalar problems (\ref{periodicdiff}) and (\ref{periodicdiffad}) agree, so that 
\begin{equation*}
\Gamma(t)=\begin{pmatrix}
-\cos(\theta t) & -\sin(\theta t)\\
\cos(\theta (t-1)) & \sin(\theta(t-1))
\end{pmatrix}
\end{equation*} 
is a fundamental matrix solution to (\ref{E:SystemAd}). 
From the discussion that follows after Proposition $\ref{P:Perp}$, we have $h\in Im(\mathcal{L})$  if and only if 
\begin{equation*}
\sum_{i=0}^{N-1}\Gamma^T(i+1)h(i)=\sum_{i=0}^{N-1}\begin{pmatrix}
-\cos(\theta (i+1)) & \cos(\theta i)\\
-\sin(\theta (i+1)) & \sin(\theta i)
\end{pmatrix}
\begin{pmatrix}
h_1(i)\\
h_2(i)
\end{pmatrix}=0,
\end{equation*}
where $v_j$ denotes the $jth$ component of a vector $v$.

 \begin{theorem}\label{T:Main2} Suppose the following conditions hold:
 
  \vspace{.25cm}
 
 \begin{enumerate}
  \item[C1.] The function $g$ is $N$-periodic in $t$ and bounded, say by $K$.
     \vspace{.25cm}
  
  \item[C2.] There exist constants $\hat{z}$ and $J>0$ such that for all $x\in\R$ with $x\geq \hat{z}$ and for every $t\in\mathbb{Z}^+$, $g(t,-x)\leq-J<0<J\leq g(t,x)$.
  \vspace{.25cm}
  
   \item[C3.] $N\theta=2\pi r$ and $\ds{{N\over gcd(r,N)}\geq\max\left\{3, {K\over J}+1\right \}}$, where $gcd(r,N)$ denotes the greatest common divisor of $r$ and $N$.
 \end{enumerate}

\noindent Then  (\ref{periodicdiff}) has a $N$-periodic solution. 
 \end{theorem}
 
 \begin{proof}
We start by defining $H_3:\R^2\times Im(I-P)\to \R^2\times Im(I-P)$ by

\begin{equation*}
\begin{split}
H_3(\alpha,x)&=\begin{pmatrix}
\ds{\sum_{i=0}^{N-1}}\Gamma^T(i+1)
\begin{pmatrix}0\\
g(i,[\Phi(i)\alpha+x(i)]_1)
\end{pmatrix}\\
x-M_p(I-Q)\mathcal{F}(\Phi(\cdot)\alpha+x)
\end{pmatrix}\\
&=
\begin{pmatrix}
\ds{\sum_{i=0}^{N-1}}\begin{pmatrix}
-\cos(\theta (i+1)) & \cos(\theta i)\\
-\sin(\theta (i+1)) & \sin(\theta i)
\end{pmatrix}
\begin{pmatrix}0\\
g(i,[\Phi(i)\alpha+x(i)]_1)
\end{pmatrix}\\
x-M_p(I-Q)\mathcal{F}(\Phi(\cdot)\alpha+x)
\end{pmatrix}\\
&=
\begin{pmatrix}
\ds{\sum_{i=0}^{N-1}}\begin{pmatrix}
\cos(\theta i)g(i,[\Phi(i)\alpha+x(i)]_1)\\
\sin(\theta i)g(i,[\Phi(i)\alpha+x(i)]_1)
\end{pmatrix}
\\
x-M_p(I-Q)\mathcal{F}(\Phi(\cdot)\alpha+x)
\end{pmatrix}.
\end{split}
\end{equation*}
From Proposition \ref{P:Lyap}  and the discussion preceding the statement of our theorem, it follows that the solutions to (\ref{E:SOper}) are precisely the zeros of $H_3$. We will show $H_3$ has a zero by showing that for some appropriately chosen $\Omega$, the Brouwer degree, $deg(H_3,\Omega, 0)$, is nonzero.

We introduce a norm which makes $\R^2\times Im(I-P)$ a Banach space.  We let 
\begin{equation*}
\norm{(\alpha,x)}=\max\{|\alpha|,\norm{x}\}.
\end{equation*}
For simplicity of notation, we will use $e^{it}$ to denote the vector $(\cos(t),\sin(t))$.  We now define, for $\ds{k=0, 1,\cdots, {N\over gcd(r,N)}-1}$, 
\begin{equation*}
A^k_s=\{\alpha\in\R^2\mid |\alpha|=s \text{ and } |<\alpha, e^{i\theta k}>|\leq\hat{z}+\norm{ME}K\}.
\end{equation*}
We will show that for large enough $s$,  $A_s^k\cap A_s^l=\emptyset$, whenever $k\neq l$. 

To see this, first note that each $\ds{k\in\left\{0, 1,\cdots, {N\over gcd(r,N)}-1\right\}}$ produces a distinct $e^{i\theta k}$.  If not, then $e^{i\theta j}=1$ for some $\ds{j\in\left\{1,\cdots, {N\over gcd(r,N)}-1\right\}}$. Thus, $\theta j=2\pi m$ for some $m$.  Since  $\ds{\theta={2\pi r\over N}}$, this gives $\ds{{2\pi rj\over N}=2\pi m}$.  Using the fact that $\ds{r\over N}={\ds{{r\over gcd(r,N)}}\over\ds{{N\over gcd(r,N)}}}$, we  get $\ds{{r\over gcd(r,N)}j={N\over gcd(r,N)}m}$.  Since $\ds{{r\over gcd(r,N)}}$ and $\ds{{N\over gcd(r,N)}}$ are relatively prime, it follows that $\ds{{N\over gcd(r,N)}}$ divides $j$, which is not the case. 
In addition, it is clear that the collection $e^{i\theta k}, k=0, 1, \cdots, N-1,$ is just $gcd(r,N)$ copies of the collection $\ds{e^{i\theta k}, k=0,1, \cdots, {N\over gcd(r,N)}-1}$ .

Further, for each $\ds{k\in\left\{0, 1,\cdots, {N\over gcd(r,N)}-1\right\}}$, $\theta k\neq \pi m$ for every odd integer $m$ .  If this was the case, then $\ds{{2\pi rk\over N}=\pi m }$, so that $Nm=2rk$, which cannot be true since $N$ and $m$ are odd. 

We now define $A=\ds{\{\theta k\mid k=0,1,\cdots,  {N\over gcd(r,N)}-1\}}$ and let $d(A, \pi \mathbb{Z})$ denote $\inf\{|a-b|\mid a\in A,  b\in \pi\mathbb{Z}\}$.  Since $A$ is finite and $\pi \mathbb{Z}$ is closed, we have that $d(A, \pi \mathbb{Z})=|\theta j-\pi m|$ for some $\ds{j\in\left\{0, 1,\cdots, {N\over gcd(r,N)}-1\right\}}$ and some $m$ in $\mathbb{Z}$.  Since, as was shown above, $\theta k\neq \pi n$ for every $\ds{k\in\left\{0, 1,\cdots, {N\over gcd(r,N)}-1\right\}}$ and every $n\in \mathbb{Z}$, we must have $d(A, \pi \mathbb{Z})>0.$

Now assume that there exists $ k$ and $l$, $k\neq l$, with $A_s^k\cap A_s^l\neq \emptyset$.  We assume, without loss of generality that $k>l$. Writing $\alpha=|\alpha|e^{i\theta_\alpha}$, this implies that both
\begin{equation*}
|\cos(\theta_\alpha-\theta k)|\leq {\hat{z}+\norm{ME}K\over s}
\end{equation*}
and 
\begin{equation*}
|\cos(\theta_\alpha-\theta l)|\leq {\hat{z}+\norm{ME}K\over s}.
\end{equation*}
 Taking $s$  large enough, we could force $|\theta_\alpha-\theta k|$ and $|\theta_\alpha-\theta l|$  to be less than $\ds{{d(A,\pi\mathbb{Z})\over 2}}$ away from an odd multiple of $\ds{{\pi\over 2}}$, say $\ds{|\theta_\alpha-\theta k-m^*{\pi\over 2}|<{d(A,\pi \mathbb{Z})\over 2}}$ and $\ds{|\theta_\alpha-\theta l-n^*{\pi\over 2}|}<{d(A,\pi \mathbb{Z})\over 2} $  .  We would then have
\begin{equation*}
\begin{split}
\ds{|\theta (k-l)-{(m^*-n^*)\over 2}\pi}|&\leq \ds{|\theta_\alpha-\theta l-n^*{\pi\over 2}|} +|\theta_\alpha-\theta k-m^*{\pi\over 2}|\\
& <{d(A,\pi \mathbb{Z})},
\end{split}
\end{equation*}
which is clearly not the case.  It is now clear that for large enough $s$, say $s\geq\hat{s}$,  $A_s^k\cap A_s^l=\emptyset$ whenever $k\neq l$. 

We now define $\Omega:=B(0,\hat{s})\times B(0,\norm{M_p(I-Q)}K)$.  We also define 

  $E:[0,1]\times\overline{\Omega}\rightarrow {\R}^{2}\times Im(I-P)$ 
  
  \noindent by 
\begin{equation*}
\begin{split}
E(\lambda, (\alpha,x))&=\begin{pmatrix}\ds{(1-\lambda)\alpha+\lambda\ds{\sum_{j=0}^{N-1}}\begin{pmatrix}
\cos(\theta j)g(j,[\Phi(i)\alpha+x(j)]_1)\\
\sin(\theta j)g(j,[\Phi(j)\alpha+x(j)]_1)
\end{pmatrix}}\\
x-\lambda M_p(I-Q)\mathcal{F}(\Phi(\cdot)\alpha+x)\end{pmatrix}\\
&=\begin{pmatrix}(1-\lambda)\alpha+\lambda \ds{\sum_{j=0}^{N-1}}e^{i\theta j}g(j,[\Phi(j)\alpha+x(j)]_1)\\
x-\lambda M_p(I-Q)\mathcal{F}(\Phi(\cdot)\alpha+x)\end{pmatrix}
\end{split}
\end{equation*}
Clearly, $E$ is a homotopy between $I$ and $H_3$; thus, the proof will be complete provided we show $0\notin E(\lambda,\partial(\Omega))$ for each $\lambda \in (0,1)$.

Now, we have that $(\alpha,x)\in \partial(\Omega)$ if and  only if 
\begin{equation*}|\alpha|=\hat{s} \text{ and } \norm{x}\leq \norm{M_p(I-Q)}K,
\end{equation*}
or
\begin{equation*}
 |\alpha|\leq \hat{s}\text{ and } \norm{x}=\norm{M_p(I-Q)}K.
\end{equation*}

We will first assume $(\alpha,x)\in\partial(\Omega)$ with $|\alpha|\leq \hat{s}\text{ and } {x}=\norm{M_p(I-Q)}K$. Then for $\lambda\in(0,1)$
\begin{equation*}
\norm{\lambda M_p(I-Q)\mathcal{F}(\Phi(\cdot)\alpha+x)}<\norm{M_p(I-Q)}K=\norm{x},
\end{equation*}
so that 
\begin{equation}\label{E:Homo2}
x-\lambda M_p(I-Q)\mathcal{F}(\Phi(\cdot)\alpha+x)\neq 0.
\end{equation}

Now we assume $(\alpha,x)\in\partial(\Omega)$ with $|\alpha|=\hat{s} \text{ and } \norm{x}\leq \norm{M_p(I-Q)}K$.  We have that
\begin{equation*}
\begin{split}
\left<\alpha, \ds{\sum_{j=0}^{N-1}}e^{i\theta j}g(j,[\Phi(j)\alpha+x(j)]_1)\right>&=\ds{\sum_{j=0}^{N-1}}\left<\alpha,e^{i\theta j}g(j,[\Phi(j)\alpha+x(j)]_1)\right>\\
&=\sum_{j=0}^{N-1}\left<\alpha,e^{i\theta j}\right>g(j,<\alpha,e^{i\theta j}>+[x(j)]_1).
\end{split}
\end{equation*}
Since $|\alpha|=\hat{s}$, $\alpha\in A_{\hat{s}}^k$ for at most one $\ds{k\in\{0, 1, \cdots, {N\over gcd(r,N)}-1\}}$, say $k^*$.  Thus,
 $|<\alpha,e^{i\theta j}>+[x(j)]_1|\geq (\hat{z}+\norm{M_p(I-Q)}K)-\norm{M_p(I-Q)}K=\hat{z}$, for any  $\ds{j\in\{0, 1, \cdots, {N\over gcd(r,N)}-1\}}$, $j\neq k^*$. It follows that for each of these $j$'s, \\$<\alpha,e^{i\theta j}>g(j,<\alpha,e^{i\theta j}>+[x(j)]_1)>(\hat{z}+\norm{ME}K)J$. Since the collection $e^{i\theta j}, j=0, 1, \cdots, N-1,$ is just $gcd(r,N)$ copies of the collection $\ds{e^{i\theta j}, j=0,1, \cdots {N\over gcd(r,N)}-1}$, this gives

$\left<\alpha, \ds{\sum_{j=0}^{N-1}}e^{i\theta j}g(j,[\Phi(j)\alpha+x(j)]_1)\right>>$
\begin{equation*}
\begin{split}
&\ds{\ds{\sum_{j\in\{0,1,\cdots,{N\over gcd(r,N)}-1 \}-\{k^*\}}}}gcd(r,N)(\hat{z}+\norm{ME}K)J\\
&\hspace{1in}+gcd(r,N)\left<\alpha,e^{i\theta k^*}\right>g(k^*,\left<\alpha,e^{i\theta k^*}\right>+[x(k^*)]_1)\\
&\geq gcd(r,N)\left(\left({N\over gcd(r,N)}-1\right)(\hat{z}+\norm{M(I-Q)}K)J-(\hat{z}+\norm{M(I-Q)}K)K\right)\\
&\geq gcd(r,N)\left({K\over J}(\hat{z}+\norm{M(I-Q)}K)J-(\hat{z}+\norm{M(I-Q)}K)K\right)=0.
\end{split}
\end{equation*}

Writing $|(1-\lambda)\alpha+\lambda \ds{\sum_{j=0}^{N-1}}e^{i\theta j}g(j,[\Phi(j)\alpha+x(j)]_1)|^2$ as 
\begin{equation*}
\begin{split}
(1-\lambda)^2||\alpha|^2&+\lambda^2\left|\ds{\sum_{j=0}^{N-1}}e^{i\theta j}g(j,[\Phi(j)\alpha+x(j)]_1)\right|\\&+(1-\lambda)\lambda\left<\alpha, \ds{\sum_{j=0}^{N-1}}e^{i\theta j}g(j,[\Phi(j)\alpha+x(j)]_1)\right>, 
\end{split}
\end{equation*}
we see 
\begin{equation}\label{E:Homo1}
|(1-\lambda)\alpha+\lambda \ds{\sum_{j=0}^{N-1}}e^{i\theta j}g(j,[\Phi(j)\alpha+x(j)]_1)|^2>0.
\end{equation}
(\ref{E:Homo2}) and (\ref{E:Homo1}) now show that $E$ is nonzero on $\partial(\Omega)$ for every $\lambda\in(0,1)$.  The proof is now complete by the invariance of the Brouwer degree under homotopy.
\end{proof}

\begin{remark}
We would like to mention that there is no analogue to Theorem \ref{T:Main2} in \cite{Etheridgesecondorder}. Throughout  \cite{Etheridgesecondorder}, the authors assume that  if $N\arccos(-{b\over2})$ is a multiple of $2\pi$, then $c\neq 1$ or $2\leq|b|$. They do so because, under these conditions, as was mentioned in the proof of Theorem \ref{T:Main1}, the solution space of the $N$-periodic linear homogeneous problem
  \begin{equation*}
y(t+2)+by(t+1)+cy(t)=0
\end{equation*} 
is at most one-dimensional.  When  $c=1$, $|b|<2$,  and $N$ is odd with $N\arccos\left(-{b\over 2}\right)$ a multiple of $2\pi$, the solution space of the $N$-periodic linear homogeneous problem is two-dimensional and in this case, the analysis of (\ref{periodicdiff}) is more delicate due to the fact that the interaction of the solution space to this homogeneous problem and the nonlinearity, $g$,  is much more complex.

\end{remark}
We now give an example which illustrates the ideas in Theorem \ref{T:Main2}.  Suppose that $N$ is prime and that $g$ is  continuous with $g(t,x)=g(x)$ for all $t\in\mathbb{Z}^+$. Suppose further, that $g$ is bounded, say $|g(x)|\leq K$, and that $\lim_{x\to \pm \infty}g(x)=\pm K$. Since $\lim_{x\to \pm \infty}g(x)=\pm K$, we have that for each $\varepsilon$, $0<\varepsilon<K$, there exists and $R_\varepsilon$ such that if $x>R_\varepsilon$, then $g(-x)<-K+\varepsilon<0<K-\varepsilon<g(x)$. Thus, we may choose $J$, of Theorem \ref{T:Main2}, to be $K-\varepsilon$ for any preferable  choice of $\varepsilon$.  Taking $\varepsilon$  smaller than $K/2$, we see that we can make $\ds{\max\left\{3, {K\over J}+1\right \}}=3$. Since $N$ is prime, $gcd(r,N)=1$ and so $\ds{{N\over gcd(r,N)}=N\geq3=\max\left\{3, {K\over J}+1\right \}}$. It follows that $C3.$, of Theorem \ref{T:Main2}, holds.  Since $C1.$ and $C2.$ trivially hold for a nonlinearity with the above mentioned properties, Theorem 4.1 gives the existence of a $N$-periodic solution to (\ref{periodicdiff}).  

We end by making mention that it is unclear whether or not the growth restrictions that allow for the existence of a solution when $Ker(\mathcal{L})$ is one-dimensional; that is, C$1.$ and C$3.$ of Theorem \ref{T:Main1}, and C$1^*.$ of Corollary \ref{C:Main1}, can be extended to the case in which $Ker(\mathcal{L})$ is two-dimensional.  It is of interest to determine if, and to what extent, the boundedness condition, C$2.$, of Theorem \ref{T:Main2} can be weakened to allow for unbounded nonlinearities, not only in the case of periodic solutions to second order difference equations, but also higher-dimensional resonance cases which arise from the study of higher-order scalar difference equations, as well as systems of difference equations.

{\bibliography{References}}{}
\bibliographystyle{plain}
 \end{document}